\documentclass[12pt]{article}

\usepackage{amssymb, amsmath, amsthm, units}
\usepackage{graphicx}
\usepackage{enumerate}

\graphicspath{ {images/} }
\DeclareGraphicsExtensions{.pdf,.jpg,.png}

\theoremstyle{plain}
\newtheorem{thm}{Theorem}
\newtheorem{lem}[thm]{Lemma}
\newtheorem{prop}[thm]{Proposition}

\newtheorem{conj}[thm]{Conjecture}

\theoremstyle{definition}
\newtheorem{defn}[thm]{Definition}

\newcommand{\girth}{\text{girth}}

\begin{document}

\title{Cycles in graphs of fixed girth with large size}

\author{J\'{o}zsef Solymosi\thanks{Research was supported by NSERC, ERC-AdG. 321104, and OTKA NK 104183 grants.} , Ching Wong\thanks{Research was supported by FYF (UBC).}}

\date{}
 
\maketitle

\begin{abstract}
Consider a family of graphs having a fixed girth and a large size. We give an optimal lower asymptotic bound on the number of even cycles of any constant length, as the order of the graphs tends to infinity.
\end{abstract}

\section{Introduction}

All graphs we consider in this article are simple graphs. We denote the size of $G$ by $e(G)$ and the order of $G$ by $v(G)$. A $j$-path in $G$ is a path of length $j$ in $G$. A $j$-cycle in $G$ is a cycle of length $j$ in $G$, and it is called an even cycle if $j$ is even. The {\em girth} of a graph $G$ is the length of the shortest cycles in $G$. For $x\in{V(G)}$, let $\Gamma^k_G(x)=\Gamma^k(x)$ denote the set of vertices of $G$ having distance exactly $k$ from the vertex $x$. 

In the following, the big-$O$ notations $f(n)=O(g(n))$ are understood as $f(n)=O(g(n))$ as $n\to\infty$, where $n$ denotes the order of a graph. The same applies to $\Theta$.

It is easy to see that a graph having large girth cannot have too many edges. The famous Erd\H{o}s girth conjecture asserts the existence of graphs of any given girth with a size of maximum possible order.

\begin{conj}[Erd\H{o}s girth conjecture]
For any positive integer $m$, there exist a constant $c>0$ depending only on $m$ and a family of graphs $\{G_n\}$ such that $v(G_n)=n$, $e(G_n)\geq{c}n^{1+1/m}$ and $\girth(G_n)>2m$.
\end{conj}

Indeed, such size is maximum by the result of Bondy and Simonovits \cite{BONSIM}, in which an explicit constant is given. They showed that a graph $G_n$ of order $n$ with $\girth(G_n)>2m$ has a size less than $100mn^{1+1/m}$.

This conjecture has been proved true for $m=1,2,3,5$. See \cite{REI}, \cite{BRO}, \cite{BEN} and \cite{WEN}. For a general $m$, Sudakov and Verstra\"ete \cite{SUDVER} showed that if such graphs exist, then they contain at least one cycle of any even length between and including $2m+2$ and $Cn$, for some constant $C>0$.

For $\ell>m$, by counting the number of $(2\ell-2m)$-paths, one can show that the number of $2\ell$-cycles in such graphs has an order not greater than $O(n^{2\ell/m})$, provided that $\deg_{G_n}(x)=\Theta(n^{1/m})$ for any vertex $x$ in $G_n$. We will see in Section \ref{reduction} that in the asymptotic case, one can assume this without loss of generality. This suggests the definition of {\em almost regularity} given in Section \ref{reduction}.

In this article, we give a lower bound on the number of $2\ell$-cycles when $\ell=O(1)$, and conclude that the number of $2\ell$-cycles is $\Theta(n^{2\ell/m})$. The precise statement is the following.

\begin{thm}
\label{ourthm}
For any real number $c>0$ and integers $M,m$ with $M>m\geq2$, there exist a constant $\alpha>0$ and an integer $N$ such that if $\{G_n\}$ is a family of graphs satisfying $v(G_n)=n$, $e(G_n)\geq{c}n^{1+1/m}$ and $\girth(G_n)>2m$, then for $n\geq{N}$ and $m+1\leq\ell\leq{M}$, the number of $2\ell$-cycles in $G_n$ is at least $\alpha{n^{2\ell/m}}$.
\end{thm}

We proceed as follows. In Section 2, we show that by adjusting the threshold $N$ in our theorem, we can further assume that the graphs have some nice properties, namely bipartite and almost regular. In Section 3, we count the number of short even cycles in $G_n$ up to length $4m$. Finally, in Section 4, we entend the argument to longer cycles, completing the proof of the main theorem.

\section{Reduction to a simpler case}
\label{reduction}
In this section, we show that it suffices to consider only bipartite graphs which are almost regular, defined as follows.

\begin{defn}
Suppose $\{G_n\}$ is a family of graphs with $v(G_n)=n$ and $\girth(G_n)>2m$, we say that $\{G_n\}$ is {\em almost regular} if there exist $c_1,c_2>0$ such that
$c_1n^{1/m}\leq\deg(x)\leq{c_2}n^{1/m}$ for any vertex $x\in{V(G_n)}$.
\end{defn}

It is a well-known fact that any graph has a bipartite subgraph with at least half of its edges. It remains to construct subgraphs whose maximum and minimum degree is of order $n^{1/m}$. To achieve this, we repeatedly apply a theorem of Bondy and Simonovits \cite{BONSIM}, which states that if an $n$-vertex graph $G_n$ has girth larger than $2m$, then $G_n$ has less than $100mn^{1+1/m}$ edges.

First we delete vertices of small degree.

\begin{lem}
\label{degreelowerbound}
For any real number $c>0$ and integer $m\geq2$, there exists a constant $\beta>0$ such that any bipartite graph $G$ of order $n$, size at least $cn^{1+1/m}$ and girth larger than $2m$ has a subgraph $H$ having at least $\beta{n}$ vertices, at least $\frac{9c}{10}n^{1+1/m}$ edges and minimum degree at least $\frac{c}{10}n^{1/m}$.
\end{lem}
\begin{proof}
Let $G=H_0$. For $i\geq1$, inductively define $H_i$ to be the subgraph of $H_{i-1}$ induced by all vertices having degree at least $\frac{c}{10}n^{1/m}$ in $H_{i-1}$. Then for all $i$ we have $e(H_i)\geq{e(H_{i+1})}$ and
\begin{equation*}\begin{split}
e(H_i)&\geq{e(H_0)}-(v(H_0)-v(H_i))\dfrac{c}{10}n^{1/m}\\
&\geq{cn^{1+1/m}}-n\dfrac{c}{10}n^{1/m}\\
&=\dfrac{9c}{10}n^{1+1/m},
\end{split}\end{equation*}
and so there exists some $j\geq0$ such that $H_i=H_j$ for all $i\geq{j}$.

Set $H=H_j$, which has girth larger than $2m$ and size at least $\frac{9c}{10}n^{1+1/m}$. Then,
\begin{equation*}
\dfrac{9c}{10}n^{1+1/m}<100m\cdot{v(H)}^{1+1/m},
\end{equation*}
or $v(H)\geq\beta{n}$, where
\begin{equation*}
\beta=\left(\dfrac{9c}{1000m}\right)^{m/(m+1)}.
\end{equation*}
\end{proof}

\begin{lem}
\label{degreeupperbound}
For every real number $c>0$ and integer $m\geq2$, there exists a constant $\gamma>0$ such that any bipartite graph $G$ of order $n$, size at least $cn^{1+1/m}$ and girth larger than $2m$ has a subgraph $H$ having at least $\frac{n}2$ vertices, at least $\frac{c}4n^{1+1/m}$ edges and maximum degree at most $\gamma{n}^{1/m}$.
\end{lem}
\begin{proof}
For any $\gamma>0$, let $S_{\gamma}$ be the set of vertices of $G$ having degree at least $\gamma{n}^{1/m}$, and let $T_{\gamma}$ be the remaining vertices of $G$. We want to find a $\gamma$ so large that $H$ can be chosen as the subgraph $G[T_\gamma]$ of $G$ induced by $T_\gamma$. It suffices to find $\gamma$ large enough so that $e(G[S_\gamma])<e(G)/4$, $e(T_\gamma,S_\gamma)<e(G)/2$ and $v(G[T_\gamma])\geq{n}/2$. 

Since both $G$ and $G[S_{\gamma}]$ have girth larger than $2m$, we can apply the result of \cite{BONSIM} twice to obtain
\begin{equation*}\begin{split}
e(G[S_{\gamma}])&<100m|S_{\gamma}|^{1+1/m}\\
&\leq100m\left(\dfrac{2e(G)}{\gamma{n}^{1/m}}\right)^{1+1/m}\\
&\leq100m\left(\dfrac{2\cdot100mn^{1+1/m}}{\gamma{n}^{1/m}}\right)^{1+1/m}\\
&=\left(\dfrac{2}{\gamma}\right)^{1+1/m}(100m)^{2+1/m}n^{1+1/m}.
\end{split}\end{equation*}
To satisfy the first condition, we choose $\gamma$ large enough so that $e(G[S_\gamma])<\frac{e(G)}{4}<\frac{100m}{4}n^{1+1/m}$, or
\begin{equation}
\label{e(S)}
\gamma>2\cdot100m\cdot4^{m/(m+1)}>400m.
\end{equation}

The second condition can be obtained via its contrapositive. Suppose $e(S_\gamma,T_\gamma)\geq\frac{c}{2}n^{1+1/m}$. Let $G_\gamma$ be the subgraph of $G$ induced by $E(S_\gamma,T_\gamma)$. Then apply Lemma \ref{degreelowerbound} to $G_\gamma$, we get a subgraph $H_\gamma$ of $G_\gamma$ having $\nu\geq\left(\frac{9c}{2000m}\right)^{m/(m+1)}n$ vertices and minimum degree at least $\frac{c}{20}n^{1/m}$. Now, we consider the subgraph $G'_\gamma$ of $G$ deleting the edges in $E(T_\gamma)$. Then, in $G'_\gamma$, every vertex in $S_\gamma\cap{V(H_\gamma)}$ still has degree at least $\gamma{n}^{1/m}$ and every vertex in $T_\gamma\cap{V(H_\gamma)}$ has degree at least $\frac{c}{20}n^{1/m}$. Note that any $m$-path in $G'_\gamma$ has at most $\lfloor{m/2}\rfloor$ internal vertices in $T_\gamma$, therefore the number of $m$-paths in $G'_\gamma$ is at least
\begin{equation*}
\dfrac12\nu\left(\dfrac{c}{20}n^{1/m}\right)^{\lfloor{m/2}\rfloor}\left(\gamma{n}^{1/m}\right)^{\lceil{m/2}\rceil}\geq\dfrac12\left(\dfrac{9c}{2000m}\right)^{m/(m+1)}\left(\dfrac{c\gamma}{20}\right)^{\lfloor{m/2}\rfloor}n^2.
\end{equation*}
But the number of $m$-paths in $G$ cannot be larger than $n^2$, since otherwise there is a pair of vertices being the endpoints of two $m$-paths, contradicting the girth of $G$ is larger than $2m$. Hence,
\begin{equation*}
\dfrac12\left(\dfrac{9c}{2000}\right)^{m/(m+1)}\left(\dfrac{c\gamma}{20}\right)^{\lfloor{m/2}\rfloor}<1,
\end{equation*}
or
\begin{equation*}
\gamma<\dfrac{20}{c}\left(2\left(\dfrac{2000}{9c}\right)^{m/(m+1)}\right)^{1/\lfloor{m/2}\rfloor}<\dfrac{40}{c}\left(\dfrac{2000}{9c}\right)^{3/(m+1)}.
\end{equation*}
Therefore, if
\begin{equation*}
\label{e(S,T)}
\gamma>\dfrac{40}{c}\left(\dfrac{2000}{9c}\right)^{3/(m+1)},
\end{equation*}
then $e(S_\gamma,T_\gamma)<\frac{c}{2}n^{1+1/m}\leq\frac{e(G)}{2}$.

Finally, since $|S_\gamma|\leq\frac{200m}{\gamma}n$, the third condition is fulfilled by \eqref{e(S)}. This finishes the proof.
\end{proof}

\section{Counting short cycles}
\label{countingshortcycles}
From now on, we suppose that $G$ is a bipartite $n$-vertex graph having at least $cn^{1+1/m}$ edges with girth larger than $2m$, such that for some constants $c_1,c_2>0$, there holds $c_1n^{1/m}\leq\deg_G(x)\leq{c_2}n^{1/m}$ for any vertex $x\in{V(G)}$.

In this section, we give a lower bound on the number of the $2\ell$-cycles in $G$, for each $m+1\leq\ell\leq2m$. 

We first sketch the idea. Let $x$ be a vertex in $G$. Suppose we have a path of odd length $k$ in $\Gamma_G^m(x)\cup\Gamma_G^{m+1}(x)$ with endpoints $w_0\in\Gamma_G^m(x)$ and $w_k\in\Gamma_G^{m+1}(x)$. For each neighbor $y$ in $\Gamma_G^m(x)$ of $w_k$, the four paths joining $x$ to $w_0$, $w_0$ to $w_k$, $w_k$ to $y$, and $y$ to $x$ form a closed walk, which contains a cycle, as shown in Figure \ref{idea}. We show in Section \ref{generic case} that generically these paths are internally disjoint, i.e. the length of the cycle is $2m+k+1$. Then we count in Section \ref{subgraph} the number of such paths and the number of neighbours of $w_k$. Finally, we obtain the desired lower bound in Section \ref{shortcycleslowerbound}.

\begin{figure}[ht!]
\centering
\includegraphics[width=80mm]{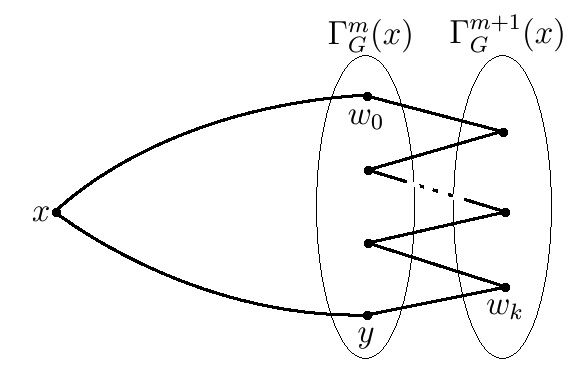}
\caption{The four paths form a closed walk, which contains a cycle}
\label{idea}
\end{figure}

\subsection{Internally disjoint closed walk}
\label{generic case}
Note that for distinct $i,j\leq{m+1}$, we know that $\Gamma^i_G(x)\cap\Gamma^j_G(x)$ is empty because $G$ is bipartite and has girth larger than $2m$. In particular, the subgraph $G_x$ of $G$ induced by the vertices $\Gamma^m_G(x)\cup\Gamma^{m+1}_G(x)$ is bipartite with bipartition $\{\Gamma^m_G(x),\Gamma^{m+1}_G(x)\}$. Hence, any path of odd length in $G_x$ has one endpoint in $\Gamma^m_G(x)$ and the other endpoint in $\Gamma^{m+1}_G(x)$.

For $1\leq{i}\leq{m}$ and any vertex $w\in\Gamma_G^i(x)$, there is a unique $(x,w)$-path $P_w^x$ of length $i$ in $G$. Note that for any two vertices $y_1,y_2\in\Gamma_G^m(x)$, the intersection of the paths $P_{y_1}^x$ and $P_{y_2}^x$ must be a path, of which $x$ is an endpoint. The following lemma guarantees that if $w_0\in\Gamma_G^m(x)$ and $w_k\in\Gamma_G^{m+1}(x)$, there is at most one neighbour $u\in\Gamma_{G_x}^1(w_k)$ so that $P_{w_0}^x$ and $P_u^x$ intersect internally, see Figure \ref{intersectinternally}.
\begin{figure}[ht!]
\centering
\includegraphics[width=80mm]{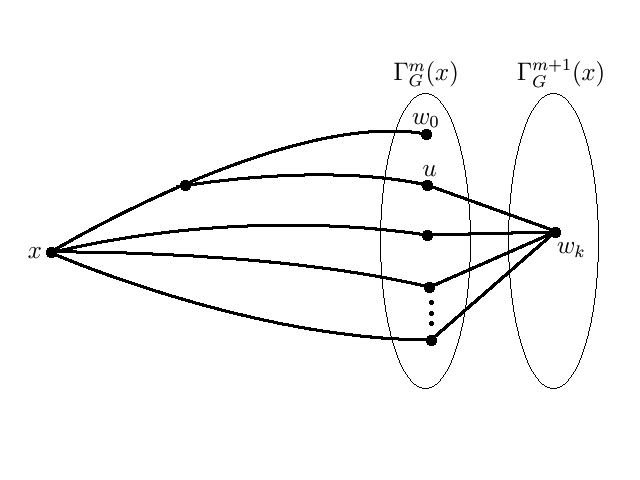}
\caption{Other neighbours of $w_k$ give internally disjoint paths}
\label{intersectinternally}
\end{figure}

\begin{lem}
\label{disjointpath}
Suppose two vertices $y_1,y_2\in\Gamma_G^m(x)$ share a common neighbour $w\in\Gamma_G^{m+1}(x)$, then the paths $P_{y_1}$ and $P_{y_2}$ are internally disjoint.
\end{lem}
\begin{proof}
\begin{figure}[ht!]
\centering
\includegraphics[width=80mm]{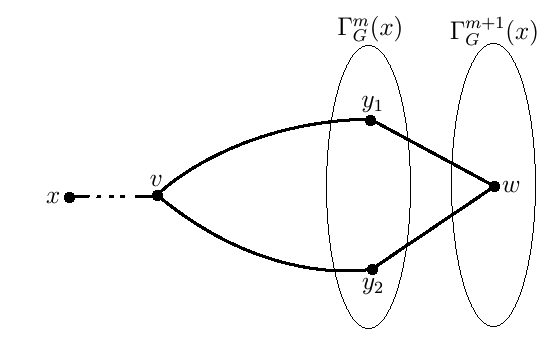}
\caption{A cycle of length at most $2m$ formed}
\label{v}
\end{figure}
Suppose the paths $P_{y_1}^x$ and $P_{y_2}^x$ intersects internally, then their intersection must be a path of length $L\geq1$, with endpoints $x$ and $v$, for some $v\in\Gamma_G^L(x)$. Thus, the union of the paths $P_{y_1}\backslash{P_v}$, $P_{y_2}\backslash{P_v}$ and the edges $(y_2,w)$, $(w,y_1)$ is a cycle of length $2(m-L)+2\leq2m$ in $G$, as in Figure \ref{v}, contradiction.
\end{proof}

Now, given a path $P=(w_0,w_1,\ldots,w_k)$ of odd length $k\leq2m-1$ in $G_x$ with $w_0\in\Gamma_G^m(x)$ and $w_k\in\Gamma_G^{m+1}(x)$. Note that $V(P)\cap\Gamma_{G_x}^1(w_k)=\{w_{k-1}\}$ as $\girth(G_n)>2m$. Let $y\in\Gamma_{G_x}^1(w_k)$. As shown in Figure \ref{cycle2m+k+1}, the four paths $P_{w_0}^x$, $P$, $(w_k,y)$ and $P_y^x$ contain a cycle of length $2m+k+1$, with at most two exceptions, namely $y=w_{k-1}$ and $y=u$.

\begin{figure}[ht!]
\centering
\includegraphics[width=80mm]{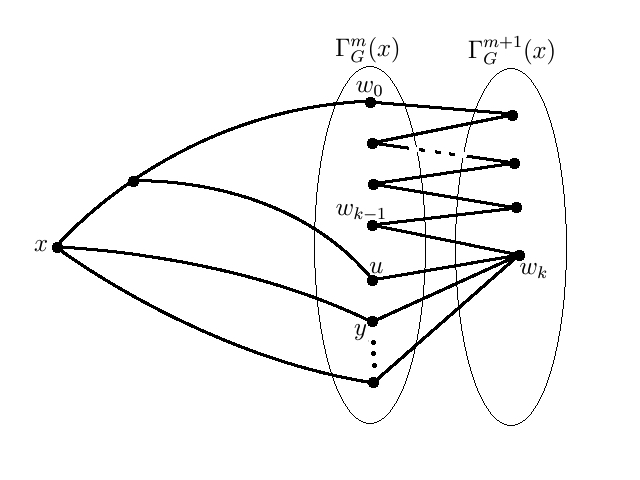}
\caption{Most neighbours of $w_k$ give cycles of length $2m+k+1$}
\label{cycle2m+k+1}
\end{figure}

\subsection{Number of paths in $G_x$}
\label{subgraph}

Note that in $G_x$, the minimum degree can be as small as $1$. Instead of counting the number of paths of a given length in $G_x$, we work with a subgraph of $G_x$ having large minimum degree. We adopt the result from Section \ref{reduction}.

It is easy to see that $v(G_x)\leq{n}$ and
\begin{equation*}\begin{split}
e(G_x)\geq\big|\Gamma_G^m(x)\big|c_1n^{1/m}\geq\big(c_1n^{1/m}\big)^{m+1}=c_1^{m+1}n^{1+1/m}.
\end{split}\end{equation*}
Using Lemma \ref{degreelowerbound}, we obtain a bipartite subgraph $H_x$ of $G_x$ having order at least $
\left(\frac{9c_1^{m+1}}{1000m}\right)^{1/(1+m)}n$, size at least $\frac{9c_1^{m+1}}{10}n^{1+1/m}$, and $\frac{c_1^{m+1}}{10}n^{1/m}\leq\deg_{H_x}(u)\leq{c_2}n^{1/m}$, for any vertex $u\in{V(H_x)}$, with bipartition $\{A_x,B_x\}$, where $A_x=\Gamma_G^m(x)\cap{H_x}$, and $B_x=\Gamma_G^{m+1}(x)\cap{H_x}$. 

\begin{lem}
\label{kpath}
Let $k$ be an odd number satisfying $1\leq{k}\leq{2m-1}$. The number of $k$-paths in $G_x$ is at least
$$\dfrac{9c_1^{(m+1)(k+1)}}{10^{k+1}c_2}n^{1+k/m}.$$
\end{lem}
\begin{proof}
The result follows from
$$|B_x|\geq\dfrac{e(H_x)}{c_2n^{1/m}}\geq\dfrac{9c_1^{m+1}}{10c_2}n$$
and that the number of $k$-paths in $G_x$ is at least
$$|B_x|\left(\dfrac{c_1^{m+1}}{10}n^{1/m}\right)^k.$$
\end{proof}

\subsection{Lower bound on the number of short cycles}
\label{shortcycleslowerbound}
The work in the preceding sections allows us to find a lot cycles in $G$. It is clear that a $2\ell$-cycle can be counted by at most $2\ell$ times as each vertex of the cycle can play the role of $x$ once. 

We are ready to give a lower bound on the number of short even cycles, up to length $4m$ in $G$.

\begin{prop}
\label{smallcycles}
Let $m$ be a positive integer. Let $G$ be a bipartite $n$-vertex graph having girth larger than $2m$ and $c_1n^{1/m}\leq\deg_G(v)\leq{c_2}n^{1/m}$ for any vertex $v\in{V(G)}$, for some $c_1,c_2>0$. Then for $m+1\leq{\ell}\leq{2m}$, the number of $2\ell$-cycles in $G$ is at least $\alpha_{\ell}{n^{2\ell/m}}$, where
\begin{equation*}
\alpha_\ell=\dfrac{9}{2\ell{c_2}}\left(\dfrac{c_1^{m+1}}{10}\right)^{2\ell-2m+1}>0.
\end{equation*}
\end{prop}
\begin{proof}
Using Lemma \ref{kpath} with $k=2\ell-2m-1$ and the observation above, the number of $2\ell$-cycles in $G$ is at least
\begin{equation*}\begin{split}
&\quad\,\,\dfrac{v(G)}{2\ell}\left(\min_{x\in{V(G)}}\deg(H_x)\right)\left(\min_{x\in{V(G)}}\text{number of }(2\ell-2m-1)\text{-paths in }H_x\right)\\
&\geq\dfrac{n}{2\ell}\left(\dfrac{c_1^{m+1}}{10}n^{1/m}\right)\left(\dfrac{9c_1^{(2\ell-2m)(m+1)}}{10^{2\ell-2m}c_2}n^{(2\ell-m-1)/m}\right)\\
&=\dfrac{9}{2\ell{c_2}}\left(\dfrac{c_1^{m+1}}{10}\right)^{2\ell-2m+1}n^{2\ell/m}.
\end{split}\end{equation*}
\end{proof}
\section{Proof of main theorem}
To count the number of longer cycles, we observe that $H_x$ has all the nice properties we wanted, namely bipartite, almost regular and large girth, we can apply Proposition \ref{smallcycles} to $H_x$ and get many short cycles in $H_x$. From them, we obtain a lot of paths in $H_x$, and each of them corresponds to many longer cycles in $G$ as in Section \ref{countingshortcycles}. These longer cycles give many longer paths in $H_x$ and again, each of these paths corresponds to many even longer cycles in $G_x$. Eventually, we have Theorem \ref{ourthm}.

For simplicity, we will assume that $m$ is even from now on. For odd $m$, one can proceed similarly.

Changing the parameters in Proposition \ref{smallcycles}, the number of $2\ell$-cycles in $H_x$ is at least
\begin{equation*}
\dfrac{9}{2\ell{c_2}}\left(\dfrac{c_1^{(m+1)^2}}{10^{m+2}}\right)^{2\ell-2m+1}n^{2\ell/m},
\end{equation*}
for $2\ell\in{L_0}:=\{3m,3m+2,3m+4,\ldots,4m\}$, and so the number of paths of length $2\ell-m\in\{2m,2m+2,2m+4,\ldots,3m\}$ in $H_x$ is at least
\begin{equation*}
\dfrac{9}{c_2}\left(\dfrac{c_1^{(m+1)^2}}{10^{m+2}}\right)^{2\ell-2m+1}n^{2\ell/m}.
\end{equation*}
Then, for $2\ell\in{L_0}$, the number of $((2\ell-m)+2m+2)$-cycles in $G$ is at least
\begin{equation*}\begin{split}
&\quad\,\dfrac{n}{2\ell+m+2}\dfrac{c_1^{2(m+1)}}{100}n^{2/m}\dfrac{9}{c_2}\left(\dfrac{c_1^{(m+1)^2}}{10^{m+2}}\right)^{2\ell-2m+1}n^{2\ell/m}\\
&=\dfrac{9}{(2\ell+m+2)c_2}\dfrac{c_1^{2(m+1)}}{100}\left(\dfrac{c_1^{(m+1)^2}}{10^{m+2}}\right)^{2\ell-2m+1}n^{(2\ell+m+2)/m},
\end{split}\end{equation*}
or for $2\ell\in{L_1}:=\{4m+2,4m+4,\ldots,5m+2\}$, the number of $2\ell$-cycles in $G$ is at least $\alpha_\ell{n^{2\ell/m}}$, where $\alpha_\ell$ is a positive constant depending on $m,c_1,c_2, \ell$ only. Repeating the same argument with the sets $L_j:=\{3m+j(m+2),3m+j(m+2)+2,\ldots,4m+j(m+2)\}$, the proof of Theorem \ref{ourthm} is completed.

\end{document}